\documentclass{article}
\usepackage{graphicx}
\usepackage{tikz}
\usepackage[margin=1in]{geometry}
\usepackage{amsmath}
\usepackage{amsthm}
\usepackage{amssymb}
\usepackage{hyperref}
\usepackage{algorithm}
\usepackage{algpseudocode}
\theoremstyle{definition}
\newtheorem{definition}{Definition}[section]
\newtheorem{remark}{Remark}[section]

\theoremstyle{plain}
\newtheorem{theorem}{Theorem}[section]
\newtheorem{corollary}[theorem]{Corollary}
\newtheorem{proposition}{Proposition}[section]
\usepackage{float}
\title{A Partition-Based Generating Function for Row-Convex Polyominoes}
\author{Vincenzo M. Scarrica}
\date{Department of Science and Technology, University of Naples Parthenope\\ 
Centro Direzionale Isola C4 - 80143, Naples, Italy\\ 
vincenzomariano.scarrica@collaboratore.uniparthenope.it}

\begin{document}

\maketitle

\section{Abstract}

An alternative generating function is proposed to enumerate row-convex 
polyominoes without internal holes on a discrete grid. The approach is 
based on integer partitions of the total area, where each partition 
corresponds to a sequence of row lengths, and the product of all 
permutations of the parts accounts for all possible horizontal alignments 
of consecutive rows. Summing over the products yields a formula for the 
total number of convex polyominoes of a given size. Numerical examples 
are provided for small areas, and the exact generating function is derived 
via a transfer series argument, establishing the asymptotic growth 
$S(N) \sim A \cdot 2^N \cos(N\theta + \phi)$ with 
$\theta = \arctan(\sqrt{7}/3)$. The method establishes a direct connection 
between integer partitions and polyomino enumeration, offering a simple 
yet effective framework for both exact and asymptotic combinatorial 
analysis. Potential applications include shape priors in discrete image 
analysis, grid-based modeling, and combinatorial generation of convex 
structures.

\section{Introduction}

A polyomino is a finite set of cells (unit squares) on the two-dimensional 
integer lattice $\mathbb{Z}^2$, such that any two cells in the set are 
connected by a path of adjacent cells sharing a common edge. Formally, 
let $P \subset \mathbb{Z}^2$ be a set of $n$ cells. $P$ is said to be 
4-connected if for every pair of cells $c_1, c_2 \in P$, there exists a 
sequence $(c_1 = d_0, d_1, \dots, d_k = c_2)$ with $d_i \in P$ for all 
$i$, such that each $d_i$ shares a side with $d_{i+1}$. This 4-connection 
ensures that the polyomino forms a single connected component without 
isolated cells \cite{golomb1954checker}.

The study of polyominoes encompasses a variety of subclasses, each defined 
by geometric or combinatorial constraints. Among the most studied are 
convex polyominoes, where the intersection with any horizontal or vertical 
line is connected, and more generally row-convex or column-convex 
polyominoes, which satisfy convexity in one direction only 
\cite{Del_Lungo_Duchi_Frosini_Rinaldi_2004}. Directed polyominoes restrict 
growth to cells added along specific directions from a root cell, 
simplifying enumeration. Other classes include $k$-convex polyominoes 
(such as L-convex \cite{CASTIGLIONE20071724} or Z-convex 
\cite{DUCHI200854}), which limit the number of changes in direction along 
monotone paths connecting any two cells, and parallelogram polyominoes 
\cite{DELEST1995503}, bounded by two monotone lattice paths. These 
distinctions are essential, as each class exhibits different combinatorial 
behaviors and challenges, and some also consider rigidity, where rotations 
or reflections yield distinct or equivalent configurations 
\cite{LEROUX1998343}.

Despite significant research, closed-form formulas for counting polyominoes 
are largely unknown \cite{bousquetmelou1999polyominoes}. Existing results 
often rely on recurrence relations or generating functions that enumerate 
polyominoes according to area or perimeter \cite{FERETIC1998173}. Even for 
convex polyominoes without internal holes, explicit closed-form formulas 
remain limited, motivating the search for alternative combinatorial 
approaches.

As a direct application of polyominoes tiling, in computer vision the 
method can be used to estimate or enumerate possible convex pixel masks of 
approximately or exactly $N$ pixels, which is relevant in scenarios with a 
fixed camera and quasi-stationary subjects, such as counting or tracking 
cellular structures and image compression 
\cite{10.1007/978-3-642-10210-3_18}. Beyond vision, convex polyomino 
enumeration may find use in statistical physics for lattice models, where 
the number of connected clusters of fixed size contributes to partition 
function calculations, or in computational chemistry for estimating 
configurations of molecules on discrete grids \cite{guttmann2009history}, 
\cite{gheorghe2008computing}. Further potential applications include tiling 
problems, grid-based procedural generation in computer graphics, and 
symbolic representations in formal languages, where convexity constraints 
simplify combinatorial analysis \cite{golomb1970tiling}.

In this work, an alternative contribution is proposed: a method to 
enumerate convex polyominoes without internal holes using the integer 
partition function, as an alternative to the classical approach introduced 
by Klarner \cite{klarner1965some} and later revisited by Hickerson 
\cite{hickerson1999counting}. Each polyomino is associated with a sequence 
of row lengths forming a partition of its total area, and the product of 
the permutations of the parts accounts for horizontal alignments of 
consecutive rows. Summing over all products provides a combinatorial 
formula for the total number of convex polyominoes of size $n$, together 
with an exact generating function and asymptotic estimate of the growth of 
this function.

The method also highlights the limitations in extending such formulas to 
concave polyominoes. In these cases, each row can contain subpartitions 
with variable spacing, making it difficult to satisfy the 4-connection 
constraint, and thereby complicating any direct combinatorial enumeration. 
Finally, simplifications using reflections and rotations are analyzed to 
reduce the number of distinct configurations, providing a practical 
framework to estimate reductions of complexity in these cases.

\section{Preliminaries}

A polyomino is defined as a finite set of cells (unit squares) on the 
two-dimensional integer lattice $\mathbb{Z}^2$, such that the set is 
4-connected, i.e., any two cells can be joined by a sequence of adjacent 
cells sharing a side. Formally, let $P \subset \mathbb{Z}^2$ be a set of 
$n$ cells. Then $P$ is 4-connected if for every pair $c_1, c_2 \in P$, 
there exists a sequence $(c_1 = d_0, d_1, \dots, d_k = c_2)$ with 
$d_i \in P$ for all $i$, such that $d_i$ and $d_{i+1}$ share a side for 
all $i$.

\begin{figure}[H]
\centering
\begin{tikzpicture}[scale=0.8]
\foreach \x in {0,1,2} {
    \foreach \y in {0,1,2} {
        \draw[fill=yellow!50, draw=black] (\x,\y) rectangle ++(1,1);
    }
}
\draw[fill=yellow!50, draw=black] (0,3) rectangle ++(1,1);
\draw[fill=yellow!50, draw=black] (2,3) rectangle ++(1,1);
\draw[fill=yellow!50, draw=black] (0,-1) rectangle ++(1,1);
\draw[fill=yellow!50, draw=black] (2,-1) rectangle ++(1,1);
\draw[fill=yellow!50, draw=black] (-1,0) rectangle ++(1,1);
\draw[fill=yellow!50, draw=black] (-1,2) rectangle ++(1,1);
\draw[fill=yellow!50, draw=black] (3,0) rectangle ++(1,1);
\draw[fill=yellow!50, draw=black] (3,2) rectangle ++(1,1);
\node at (1, -2) {Convex polyomino not row-convex, 17 cells};
\end{tikzpicture}
\caption{A 17-cell convex polyomino with corner extensions showing lack 
of row-convexity.}
\label{fig:convex_not_row_convex}
\end{figure}

Polyominoes can be classified according to geometric constraints. In this 
work, we focus on the following classes:

\begin{itemize}
    \item \emph{Convex polyominoes}: a polyomino $P$ is convex if the 
    intersection of $P$ with any horizontal or vertical line is connected. 
    Equivalently, for any two cells $c_1, c_2 \in P$, there exists a 
    monotone path contained in $P$ connecting them. In 
    Fig.~\ref{fig:convex_not_row_convex}, a convex polyomino example with 
    17 cells is shown.
    \item \emph{Row-convex / Column-convex polyominoes}: a polyomino is 
    row-convex (column-convex) if each row (column) contains a contiguous 
    sequence of cells. Even though the polyomino in 
    Fig.~\ref{fig:convex_not_row_convex} is convex, it is not row-convex 
    or column-convex.
    \item \emph{Polyominoes without holes}: $P$ is simply connected if 
    every cell in the bounding rectangle that lies between two cells of $P$ 
    is also in $P$, ensuring no internal voids.
\end{itemize}

Figure~\ref{fig:tetrominoes} illustrates examples of polyominoes of area 
4, highlighting convex shapes of fixed size and showing how row and column 
convexity appear in simple configurations.

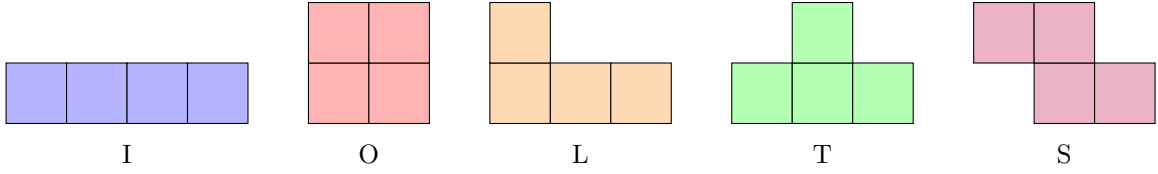
\begin{figure}[H]
\centering
\begin{tikzpicture}[scale=0.8]
\draw[fill=blue!30] (0,0) rectangle (1,1);
\draw[fill=blue!30] (1,0) rectangle (2,1);
\draw[fill=blue!30] (2,0) rectangle (3,1);
\draw[fill=blue!30] (3,0) rectangle (4,1);
\node at (2,-0.5) {I};
\draw[fill=red!30] (5,0) rectangle (6,1);
\draw[fill=red!30] (6,0) rectangle (7,1);
\draw[fill=red!30] (5,1) rectangle (6,2);
\draw[fill=red!30] (6,1) rectangle (7,2);
\node at (6,-0.5) {O};
\draw[fill=orange!30] (8,0) rectangle (9,1);
\draw[fill=orange!30] (9,0) rectangle (10,1);
\draw[fill=orange!30] (10,0) rectangle (11,1);
\draw[fill=orange!30] (8,1) rectangle (9,2);
\node at (9.5,-0.5) {L};
\draw[fill=green!30] (12,0) rectangle (13,1);
\draw[fill=green!30] (13,0) rectangle (14,1);
\draw[fill=green!30] (14,0) rectangle (15,1);
\draw[fill=green!30] (13,1) rectangle (14,2);
\node at (13.5,-0.5) {T};
\draw[fill=purple!30] (17,0) rectangle (18,1);
\draw[fill=purple!30] (18,0) rectangle (19,1);
\draw[fill=purple!30] (16,1) rectangle (17,2);
\draw[fill=purple!30] (17,1) rectangle (18,2);
\node at (17.5,-0.5) {S};
\end{tikzpicture}
\caption{The five tetrominoes (polyominoes of area 4) commonly known from 
Tetris: I, O, L, T, and S.}
\label{fig:tetrominoes}
\end{figure}

We denote by $S(n)$ the total number of row-convex polyominoes without 
holes of area $n$. Later, we will introduce a combinatorial formula for 
$S(n)$ based on integer partitions and products of row lengths, providing 
both exact enumeration for small $n$ and an asymptotic estimate for 
large $n$.

\begin{figure}[H]
\centering
\begin{tikzpicture}[scale=0.8]
\foreach \x in {0,1,2}{
    \foreach \y in {0,1,2}{
        \ifnum\x=1
            \ifnum\y=1
            \draw[fill=white] (\x,\y) rectangle ++(1,1);
            \else
            \draw[fill=cyan!30] (\x,\y) rectangle ++(1,1);
            \fi
        \else
        \draw[fill=cyan!30] (\x,\y) rectangle ++(1,1);
        \fi
    }
}
\node at (1.5,-0.5) {Area 8, with internal hole};
\end{tikzpicture}
\caption{Example of the first polyomino with an internal hole (area 8).}
\label{fig:polyomino_hole}
\end{figure}

\textbf{Remark.} Convex polyominoes exist only for $N \le 7$. From 
$N = 8$, concavity can appear, and Figure~\ref{fig:polyomino_hole} shows 
the first example of a polyomino with an internal hole.

\begin{figure}[H]
\centering
\begin{tikzpicture}[scale=0.5]
\begin{scope}[shift={(0,0)}]
\foreach \x/\ylist in {0/{0,1,2},1/{0,1,2,3},2/{-1,0,1,2,3},3/{0,1,2,3},4/{0,1,2},5/{0,1,2,3},6/{0,1,2},7/{0,1}}{
    \foreach \y in \ylist {
        \draw[fill=cyan!30] (\y,-\x) rectangle ++(1,1);
    }
}
\node at (2,-9.5) {Convex, area 20};
\end{scope}
\begin{scope}[shift={(14,0)}]
\foreach \x/\ylist in {0/{-1,0,1},1/{0,1,2,3},2/{0,2,3},3/{-1,0,1,3,4},4/{0,1,2,4},5/{0,2,3},6/{0,1,2},7/{0,1}}{
    \foreach \y in \ylist {
        \draw[fill=cyan!30] (\y,-\x) rectangle ++(1,1);
    }
}
\node at (4,-9.5) {Concave, area 20, with holes};
\end{scope}
\end{tikzpicture}
\caption{Comparison of two polyominoes of area 20: left, convex without 
holes; right, concave with holes.}
\label{fig:convex_vs_concave_frastagliated}
\end{figure}

\noindent
\textbf{Comment.} Figure~\ref{fig:convex_vs_concave_frastagliated} 
compares two polyominoes of area 20. The left polyomino is convex: all 
rows and columns contain contiguous blocks without internal holes. The 
right polyomino is concave and contains $4$ internal holes, clearly 
showing non-contiguous rows and concave regions.

\begin{definition}[Partition Function]
Let $N$ be a positive integer. A \emph{partition} of $N$ is a sequence 
of positive integers $(\lambda_1, \lambda_2, \dots, \lambda_k)$ such that
\[
\lambda_1 \ge \lambda_2 \ge \dots \ge \lambda_k \ge 1 \quad \text{and} 
\quad \sum_{i=1}^{k} \lambda_i = N.
\]
The \emph{partition function} $P(N)$ counts the total number of such 
partitions of $N$ \cite{erdos1941distribution}.
\end{definition}

\begin{remark}
By convention, we set $P(1) = 1$, representing the unique partition of 1.
\end{remark}

\begin{proposition}[Recurrence Relation]
For $N \ge 2$, the partition function satisfies the recurrence
\[
P(N) = \sum_{K=1}^{N-1} P(K)\,P(N-K).
\]
\end{proposition}

\begin{figure}[H]
\centering
\begin{tikzpicture}[scale=0.8]
\begin{scope}[shift={(0,0)}]
\foreach \y in {0,1,2,3,4,5} {
    \draw[fill=cyan!30, draw=black] (\y,0) rectangle ++(1,1);
}
\node at (2.5,-1.2) {A) Maximum-width 6-polyomino};
\end{scope}
\begin{scope}[shift={(8,0)}]
\foreach \x in {0,1} {
    \draw[fill=red!50, draw=black] (\x,0) rectangle ++(1,1);
}
\draw[fill=green!50, draw=black] (2,0) rectangle ++(1,1);
\foreach \x in {3,4} {
    \draw[fill=blue!50, draw=black] (\x,0) rectangle ++(1,1);
}
\draw[fill=magenta!50, draw=black] (5,0) rectangle ++(1,1);
\node at (3,-1.2) {B) Partitions 2+1+2+1};
\end{scope}
\begin{scope}[shift={(18,0)}]
\foreach \x in {0,1} {
    \draw[fill=red!50, draw=black] (\x-1,0) rectangle ++(1,1);
}
\draw[fill=green!50, draw=black] (0,-1) rectangle ++(1,1);
\foreach \x in {0,1} {
    \draw[fill=blue!50, draw=black] (\x,-2) rectangle ++(1,1);
}
\draw[fill=magenta!50, draw=black] (0,-3) rectangle ++(1,1);
\node at (0.5,-4.2) {C) Stacked polyomino};
\end{scope}
\end{tikzpicture}
\caption{Applying a partition to a polyomino. A: original 6-cell row; 
B: split into partitions $2+1+2+1$; C: stacked blocks form the polyomino.}
\label{fig:polyomino_partition_example}
\end{figure}
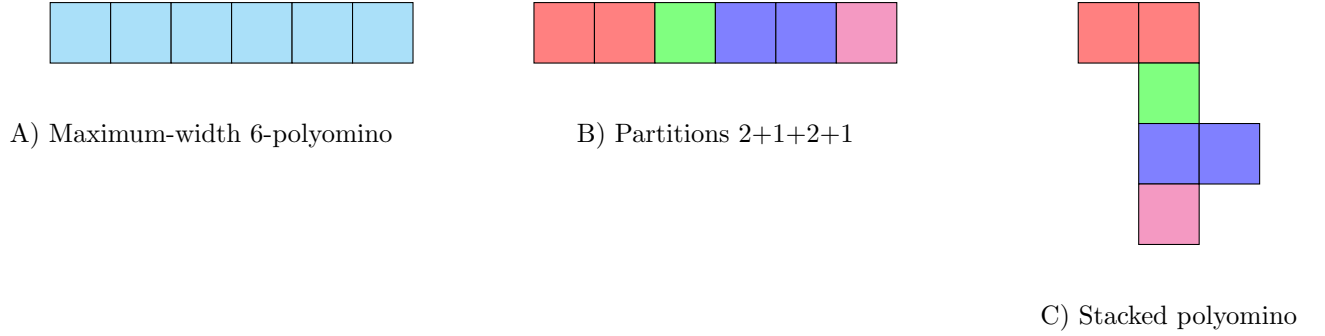

\begin{definition}[Shift of Two Rows]
Let $R_1$ and $R_2$ be two horizontal polyominoes of lengths $L_1$ and 
$L_2$. The \emph{shift} of $R_2$ over $R_1$ consists of all positions 
such that at least one square of $R_2$ touches a square of $R_1$. The 
number of possible shifts is
\[
S_2 = L_1 + L_2 - 1.
\]
\end{definition}

\begin{figure}[H]
\centering
\begin{tikzpicture}[scale=0.8]
\begin{scope}[shift={(0,0)}]
\foreach \x in {0,1,2,3} { \draw[fill=cyan!30, draw=black] (\x,0) rectangle ++(1,1); }
\foreach \x in {0,1,2} { \draw[fill=red!50, draw=black] (\x-2,1) rectangle ++(1,1); }
\node at (1.5,-0.8) {Original};
\end{scope}
\draw[->, thick] (6,0.5) -- (8,0.5) node[midway, above] {Shift};
\begin{scope}[shift={(10,0)}]
\foreach \x in {0,1,2,3} { \draw[fill=cyan!30, draw=black] (\x,0) rectangle ++(1,1); }
\foreach \x in {0,1,2} { \draw[fill=red!50, draw=black] (\x+3,1) rectangle ++(1,1); }
\node at (1.5,-0.8) {Shifted};
\end{scope}
\end{tikzpicture}
\caption{Two-row shift. Number of possible shifts: $S_2 = 4 + 3 - 1 = 6$.}
\label{fig:two_row_shift_correct}
\end{figure}
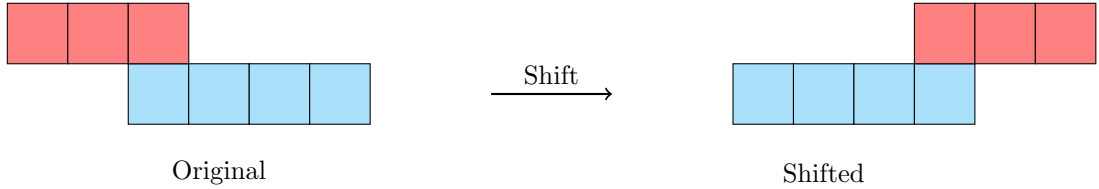

\noindent
\textbf{Comment:} The bottom row has length $L_1 = 4$ and the top row 
$L_2 = 3$, giving exactly $6$ positions, confirming 
$S_2 = L_1 + L_2 - 1 = 6$.

\begin{definition}[Shift of Three Rows]
Let $R_1$, $R_2$, $R_3$ have lengths $L_1$, $L_2$, $L_3$. We fix the 
first row $R_1$ and shift $R_2$ over it. Then, for each configuration of 
$R_2$, we shift $R_3$ over the effective base formed by $R_1$ and $R_2$. 
The total number of maximal shifts is
\[
S_3 = (L_1 + L_2 - 1) \cdot (L_2 + L_3 - 1).
\]
\end{definition}

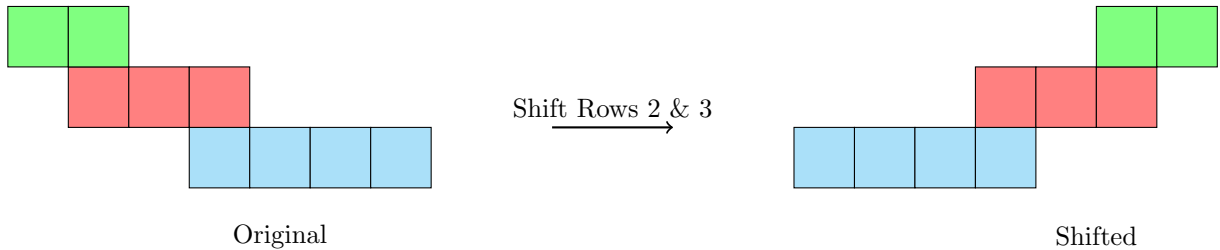
\begin{figure}[H]
\centering
\begin{tikzpicture}[scale=0.8]
\begin{scope}[shift={(0,0)}]
\foreach \x in {0,1,2,3} { \draw[fill=cyan!30, draw=black] (\x,0) rectangle ++(1,1); }
\foreach \x in {0,1,2} { \draw[fill=red!50, draw=black] (\x-2,1) rectangle ++(1,1); }
\foreach \x in {0,1} { \draw[fill=green!50, draw=black] (\x-3,2) rectangle ++(1,1); }
\node at (1.5,-0.8) {Original};
\end{scope}
\draw[->, thick] (6,1) -- (8,1) node[midway, above] {Shift Rows 2 \& 3};
\begin{scope}[shift={(10,0)}]
\foreach \x in {0,1,2,3} { \draw[fill=cyan!30, draw=black] (\x,0) rectangle ++(1,1); }
\foreach \x in {3,4,5} { \draw[fill=red!50, draw=black] (\x,1) rectangle ++(1,1); }
\foreach \x in {5,6} { \draw[fill=green!50, draw=black] (\x,2) rectangle ++(1,1); }
\node at (5,-0.8) {Shifted};
\end{scope}
\end{tikzpicture}
\caption{Three-row maximal shift. Total configurations: 
$(L_1+L_2-1)\cdot(L_2+L_3-1)=6\cdot4=24$.}
\label{fig:three_row_shift_precise}
\end{figure}

\noindent
\textbf{Comment:} Bottom row $L_1=4$, middle row $L_2=3$ gives 
$L_1+L_2-1=6$ positions; top row $L_3=2$ gives $L_2+L_3-1=4$ positions; 
total: $S_3 = 6\cdot 4 = 24$.

\begin{definition}[Generalized Shift for $r$ Rows]
Let $R_1,\dots,R_r$ be $r$ horizontal polyominoes. The total number of 
maximal configurations is
\[
S_r = \prod_{i=1}^{r-1} (L_i + L_{i+1}-1).
\]
\end{definition}

\begin{definition}[Permutation of Polyomino Rows]
Let $P$ be a row-convex polyomino composed of $r$ horizontal rows 
$(R_1, R_2, \dots, R_r)$. A \emph{permutation of the rows of $P$} is any 
reordering $(R_{\sigma(1)}, R_{\sigma(2)}, \dots, R_{\sigma(r)})$ where 
$\sigma$ is a permutation of $\{1,2,\dots,r\}$. Each permutation produces 
a new configuration preserving row lengths but possibly changing the 
overall shape.
\end{definition}

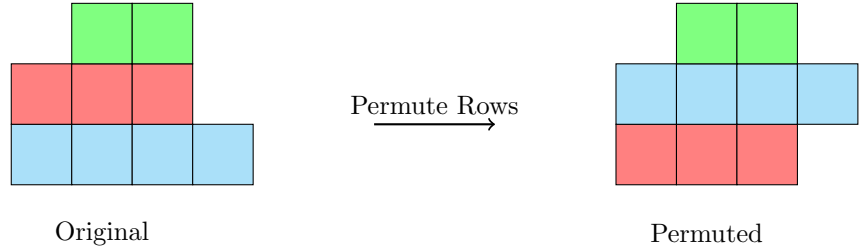
\begin{figure}[H]
\centering
\begin{tikzpicture}[scale=0.8]
\begin{scope}[shift={(0,0)}]
\foreach \x in {0,1,2,3} { \draw[fill=cyan!30, draw=black] (\x,0) rectangle ++(1,1); }
\foreach \x in {0,1,2} { \draw[fill=red!50, draw=black] (\x,1) rectangle ++(1,1); }
\foreach \x in {0,1} { \draw[fill=green!50, draw=black] (\x+1,2) rectangle ++(1,1); }
\node at (1.5,-0.8) {Original};
\end{scope}
\draw[->, thick] (6,1) -- (8,1) node[midway, above] {Permute Rows};
\begin{scope}[shift={(10,0)}]
\foreach \x in {0,1,2} { \draw[fill=red!50, draw=black] (\x,0) rectangle ++(1,1); }
\foreach \x in {0,1,2,3} { \draw[fill=cyan!30, draw=black] (\x,1) rectangle ++(1,1); }
\foreach \x in {0,1} { \draw[fill=green!50, draw=black] (\x+1,2) rectangle ++(1,1); }
\node at (1.5,-0.8) {Permuted};
\end{scope}
\end{tikzpicture}
\caption{Permutation of the rows of a row-convex polyomino. Left: rows of 
lengths $4,3,2$. Right: permuted order $3,4,2$.}
\label{fig:polyomino_row_permutation}
\end{figure}

\section{Enumeration of Convex Polyominoes with Distinct Permutations}

\begin{definition}[Partition-induced sequences with distinct permutations]
Let $N \in \mathbb{N}_{\ge 1}$ and $\mathcal{P}(N)$ be the set of integer 
partitions
\[
\boldsymbol{\lambda} = (\lambda_1, \lambda_2, \dots, 
\lambda_{\ell(\boldsymbol{\lambda})}), \quad 
\sum_{i=1}^{\ell(\boldsymbol{\lambda})} \lambda_i = N.
\]
For each partition $\boldsymbol{\lambda}$, define the 
\emph{permutation factor} $\Phi(\boldsymbol{\lambda})$ as
\[
\Phi(\boldsymbol{\lambda}) =
\begin{cases}
1, & \text{if all parts of } \boldsymbol{\lambda} \text{ are equal},\\[1mm]
\dfrac{\ell(\boldsymbol{\lambda})!}{\prod_v m_v!}, & \text{otherwise,}
\end{cases}
\]
where $m_v$ is the multiplicity of value $v$ in $\boldsymbol{\lambda}$.
\end{definition}

\begin{theorem}[Enumeration Formula]
\label{teo:teorema}
The number of row-convex polyominoes with $N$ cells is
\[
S(N) = \sum_{\boldsymbol{\lambda} \in \mathcal{P}(N)} 
\Phi(\boldsymbol{\lambda}) \prod_{i=1}^{\ell(\boldsymbol{\lambda})-1} 
(\lambda_i + \lambda_{i+1} - 1).
\]
\end{theorem}

\begin{proof}
\textbf{(1) Row decomposition.} Every row-convex polyomino decomposes into 
a sequence of contiguous rows with lengths $\boldsymbol{\lambda}$.

\textbf{(2) Distinct row orderings.} Permuting rows with distinct lengths 
produces a different polyomino, whereas permuting identical rows does not. 
This is captured by $\Phi(\boldsymbol{\lambda})$.

\textbf{(3) Horizontal shifts.} For each consecutive pair 
$(\lambda_i, \lambda_{i+1})$, there are $\lambda_i + \lambda_{i+1} - 1$ 
independent horizontal alignments.

\textbf{(4) Total enumeration.} Multiplying the shift product by 
$\Phi(\boldsymbol{\lambda})$ gives all configurations corresponding to 
$\boldsymbol{\lambda}$. Summing over all $\boldsymbol{\lambda} \in 
\mathcal{P}(N)$ enumerates all row-convex polyominoes without overcounting.
\end{proof}

\begin{corollary}[Row- vs Column-Convex Polyominoes]
\label{cor:rowcol}
$S_{\mathrm{row}}(N) = S_{\mathrm{col}}(N)$, and the correspondence is 
established by a rotation of $90^\circ$.
\end{corollary}

\begin{proof}
The map $\Psi$ defined by $90^\circ$ clockwise rotation sends row-convex 
polyominoes to column-convex polyominoes. It is well-defined (contiguous 
rows become contiguous columns), injective (distinct polyominoes rotate to 
distinct ones), and surjective (rotating $90^\circ$ counterclockwise 
inverts it). Hence $\Psi$ is a bijection.
\end{proof}

\section{Symmetry and Partition Reduction}

A natural question is whether we can reduce the total count by identifying 
configurations that are mirror images of each other. Given a row-convex 
polyomino, its reflection with respect to the vertical axis produces 
another valid polyomino \cite{LEROUX1998343}.

\begin{itemize}
    \item If all rows have even length, every reflected polyomino is 
    distinct from its original, and the total count can be exactly halved.
    \item If two contiguous rows have mixed parity, the reflection still 
    produces a distinct polyomino, and the halving argument approximately 
    holds.
    \item If two rows are of odd length, the number of admissible 
    horizontal shifts is odd, and there exists exactly one polyomino 
    symmetric under reflection.
\end{itemize}

\noindent Since it is difficult to determine from the enumeration formula 
which sequences correspond to symmetric polyominoes, we bound the number 
of distinct polyominoes up to reflection:
\[
\frac{S(N)}{2} \le S_{\text{distinct}}(N) \le 
\sum_{\boldsymbol{\lambda} \in \mathcal{P}(N)} 
\Phi(\boldsymbol{\lambda}) \prod_{i=1}^{\ell(\boldsymbol{\lambda})-1} 
(\lambda_i + \lambda_{i+1}),
\]
where the upper bound follows by replacing each factor 
$(\lambda_i + \lambda_{i+1} - 1)$ with $(\lambda_i + \lambda_{i+1})$, 
which accounts for the at most one symmetric configuration per pair of 
odd-length rows.

\section{Generating Function and Exact Asymptotics}

We now derive the generating function of $S(N)$ directly from 
Theorem~\ref{teo:teorema} and use it to establish the exact asymptotic 
behavior of the sequence.

\subsection{Reduction to ordered compositions}

The key observation is that every distinct permutation of a partition is 
an ordered composition, and every ordered composition appears exactly once 
among the distinct permutations of its underlying partition. Therefore 
Theorem~\ref{teo:teorema} is equivalent to
\[
S(N) = \sum_{k=1}^{N} \sum_{\substack{(\pi_1,\ldots,\pi_k)\\ 
\pi_i \geq 1,\ \sum_{i=1}^k \pi_i = N}} 
\prod_{i=1}^{k-1}(\pi_i + \pi_{i+1} - 1).
\]

\subsection{Transfer series}

For $m \geq 1$, define $F_m(x)$ as the generating function of ordered 
compositions ending with part $m$, weighted by the product of sliding 
factors:
\[
F_m(x) = \sum_{N \geq m} \left(
\sum_{\substack{(\pi_1,\ldots,\pi_k)\\ 
\pi_k = m,\ \sum \pi_i = N}} 
\prod_{i=1}^{k-1}(\pi_i+\pi_{i+1}-1)
\right) x^N.
\]
Each such composition is either the single-part sequence $(m)$, or a 
composition ending in some $l \geq 1$ followed by $m$, giving
\[
F_m(x) = x^m + x^m \sum_{l \geq 1}(l + m - 1)F_l(x).
\]
Defining
\[
S(x) = \sum_{m \geq 1} F_m(x), \qquad R(x) = \sum_{m \geq 1} m\, F_m(x),
\]
and using $\sum_{l\geq 1}(l+m-1)F_l(x) = R(x) + (m-1)S(x)$, we obtain
\[
F_m(x) = x^m\bigl(1 + R(x) + (m-1)S(x)\bigr). \tag{$*$}
\]

\subsection{Closing the system}

\textbf{Equation for $S(x)$.} Summing $(*)$ over $m \geq 1$ and using 
$\sum_{m\geq 1} x^m = \frac{x}{1-x}$ and 
$\sum_{m\geq 1} m x^m = \frac{x}{(1-x)^2}$:
\[
S = (1+R)\frac{x}{1-x} + S\frac{x^2}{(1-x)^2}.
\]
Collecting $S$:
\[
S = \frac{x(1-x)(1+R)}{1-2x}. \tag{1}
\]

\textbf{Equation for $R(x)$.} Multiplying $(*)$ by $m$, summing over 
$m \geq 1$, and using $\sum_{m\geq 1} m(m-1)x^m = \frac{2x^2}{(1-x)^3}$:
\[
R = \frac{x}{(1-x)^2}(1+R) + \frac{2x^2}{(1-x)^3}S. \tag{2}
\]

\textbf{Substitution.} From (1): 
$1+R = \frac{S(1-2x)}{x(1-x)}$.
Substituting into (2):
\[
\frac{S(1-2x)}{x(1-x)} - 1 
= \frac{S(1-2x)}{(1-x)^3} + \frac{2x^2}{(1-x)^3}S.
\]
Collecting $S$ over the common denominator $x(1-x)^3$ and expanding:
\begin{align*}
(1-2x)(1-x)^2 - x(1-2x+2x^2) 
&= 1-4x+5x^2-2x^3 - x+2x^2-2x^3 \\
&= 1-5x+7x^2-4x^3.
\end{align*}
Therefore $S(x)\cdot\frac{1-5x+7x^2-4x^3}{x(1-x)^3} = 1$, giving the 
following result.

\begin{theorem}[Generating Function]
\label{thm:gf}
The generating function of the number of row-convex polyominoes is
\[
G(x) = \sum_{N \geq 1} S(N)\, x^N = \frac{x(1-x)^3}{1-5x+7x^2-4x^3}.
\]
\end{theorem}

\noindent This coincides with the generating function of column-convex 
polyominoes derived by Klarner \cite{klarner1965some} and Hickerson 
\cite{hickerson1999counting}, and together with 
Corollary~\ref{cor:rowcol} provides a complete proof that our 
combinatorial formula enumerates exactly the same objects. 
Table~\ref{tab:comparison} confirms this identity numerically for 
small $N$.

\begin{table}[h!]
\centering
\begin{tabular}{c|c|c}
$N$ & $S(N)$ This work & Hickerson \cite{hickerson1999counting} \\
\hline
1  & 1      & 1      \\
2  & 2      & 2      \\
3  & 6      & 6      \\
4  & 19     & 19     \\
5  & 61     & 61     \\
6  & 196    & 196    \\
7  & 629    & 629    \\
8  & 2017   & 2017   \\
9  & 6466   & 6466   \\
10 & 20727  & 20727  \\
11 & 66441  & 66441  \\
12 & 212980 & 212980 \\
\end{tabular}
\caption{Comparison of row-convex polyomino counts (including permutation 
factor $\Phi$) with Hickerson \cite{hickerson1999counting}.}
\label{tab:comparison}
\end{table}

\subsection{Asymptotic behavior}

\begin{theorem}[Asymptotic Growth]
\label{thm:asymptotic}
Let $S(N)$ denote the number of row-convex polyominoes with $N$ cells. 
Then
\[
S(N) \sim A \cdot 2^N \cos(N\theta + \phi), \quad N \to \infty,
\]
where $\theta = \arctan\!\left(\dfrac{\sqrt{7}}{3}\right)$ and the 
constants $A > 0$, $\phi \in \mathbb{R}$ are determined by the residues 
of $G(x)$ at the dominant poles.
\end{theorem}

\begin{proof}
The denominator of $G(x)$ satisfies $4x^3-7x^2+5x-1=0$. One root is 
$x=1$, which is removable since the numerator $x(1-x)^3$ vanishes there 
to order $3$. The remaining two roots are
\[
x_{2,3} = \frac{3 \pm i\sqrt{7}}{8},
\]
with modulus
\[
|x_{2,3}| = \frac{\sqrt{3^2+(\sqrt{7})^2}}{8} = \frac{\sqrt{16}}{8} 
= \frac{1}{2}.
\]
These are simple poles of $G(x)$ and constitute the dominant 
singularities. By the transfer lemma for simple poles 
\cite{flajolet2009analytic}, the coefficients satisfy
\[
[x^N] G(x) = 
2\,\mathrm{Re}\!\left[\mathrm{Res}(G, x_2)\cdot x_2^{-N-1}\right]
\sim A \cdot 2^N \cos(N\theta + \phi),
\]
where $\theta = \arg(x_2) = \arctan\!\left(\dfrac{\sqrt{7}}{3}\right)$ 
and $A$, $\phi$ are determined by the residue at $x_2$.
\end{proof}

\begin{remark}
The dominant growth rate is $2^N$, reflecting purely exponential growth 
with base $2$. This contrasts with the Hardy--Ramanujan estimate for 
unrestricted integer partitions, which grows sub-exponentially as 
$\exp\!\left(\pi\sqrt{2N/3}\right)$. The exponential growth of $S(N)$ 
arises from the sliding weights $(\pi_i + \pi_{i+1}-1)$, which dominate 
the partition enumeration.
\end{remark}

\paragraph{Implementation and Verification.}
Algorithm~\ref{alg:row_convex} implements the proposed formula in Python, 
systematically enumerating all integer partitions of $N$ and computing the 
corresponding products to obtain $S(N)$. The computed values have been 
compared with the known sequence of column-convex polyominoes listed in 
OEIS \href{https://oeis.org/A001169}{A001169}, confirming the correctness 
of the implementation.

\begin{algorithm}[H]
\caption{Computation of $S(N)$ for Row-Convex Polyominoes}
\label{alg:row_convex}
\begin{algorithmic}[1]
\Function{Partitions}{$n, \text{max}$}
    \If{$n = 0$} \State \Return \{ empty sequence \} \EndIf
    \If{$\text{max}$ is undefined or $\text{max} > n$} 
        \State $\text{max} \gets n$
    \EndIf
    \For{$i = \text{max}$ down to $1$}
        \For{each $\text{tail}$ in \Call{Partitions}{$n-i, i$}}
            \State \textbf{yield} $[i] + \text{tail}$
        \EndFor
    \EndFor
\EndFunction
\Function{DistinctPermutations}{$\text{seq}$}
    \State \Return all distinct permutations of $\text{seq}$
\EndFunction
\Function{$S_{\text{row-convex}}$}{$N$}
    \State $\text{total} \gets 0$
    \For{each partition $p$ in \Call{Partitions}{$N$}}
        \If{all elements in $p$ are equal}
            \State $\text{sequences} \gets \{ p \}$
        \Else
            \State $\text{sequences} \gets$ \Call{DistinctPermutations}{$p$}
        \EndIf
        \For{each sequence $\pi$ in $\text{sequences}$}
            \State $\text{prod} \gets 1$
            \For{$i = 1$ to $\text{length}(\pi)-1$}
                \State $\text{prod} \gets \text{prod} \cdot 
                (\pi_i + \pi_{i+1} - 1)$
            \EndFor
            \State $\text{total} \gets \text{total} + \text{prod}$
        \EndFor
    \EndFor
    \State \Return $\text{total}$
\EndFunction
\end{algorithmic}
\end{algorithm}

\section{Limitations and Extension to Concave Polyominoes}

\begin{definition}[Sub-Partition of a Polyomino Row]
Let $R$ be a row of a concave polyomino \cite{barequet2024totally} of 
length $L$. A \emph{sub-partition} of $R$ is a sequence of positive 
integers $\boldsymbol{\mu} = (\mu_1, \mu_2, \dots, \mu_k)$ with 
$\sum_{i=1}^{k} \mu_i \le L$, representing contiguous blocks of filled 
cells separated by gaps. Each $\mu_i$ corresponds to a maximal connected 
component in $R$.
\end{definition}

In convex polyominoes, each row is fully connected, reducing the 
enumeration to integer partitions with horizontal shifts. In concave 
polyominoes, rows can split into multiple components, and the allowed 
alignments must satisfy
\[
\forall i \in \{1, \dots, r-1\}, \quad 
\text{each block in } \boldsymbol{\mu}^{(i+1)} 
\text{ touches at least one block in } \boldsymbol{\mu}^{(i)}.
\]
Enumerating all configurations consistent with 4-connectivity is 
combinatorially intractable, so a direct generalization of the present 
approach to concave polyominoes is not currently feasible.

\section{Conclusions}

In this work, we have presented an alternative combinatorial method to 
enumerate row-convex polyominoes using the integer partition function. 
Each polyomino is associated with a sequence of row lengths forming a 
partition of its total area, and the product of consecutive row lengths 
accounts for all possible horizontal alignments. By including the 
permutation factor $\Phi(\boldsymbol{\lambda})$, we ensure that distinct 
arrangements of rows are properly counted.

The exact agreement with the classical enumeration of Hickerson 
\cite{hickerson1999counting} and Klarner \cite{klarner1965some} is fully 
explained by Theorem~\ref{thm:gf}, which establishes that the generating 
function derived from our combinatorial formula coincides identically with 
the known generating function of column-convex polyominoes. As a 
consequence, Theorem~\ref{thm:asymptotic} yields the exact asymptotic 
behavior $S(N) \sim A \cdot 2^N \cos(N\theta + \phi)$ with 
$\theta = \arctan(\sqrt{7}/3)$, confirming purely exponential growth with 
base $2$.

Furthermore, we discussed reductions in the total count via symmetry and 
rotations. Identifying mirror-image configurations allows for a reduction 
of the enumeration effort, and the row-column equivalence established in 
Corollary~\ref{cor:rowcol} confirms that the same growth applies to 
column-convex polyominoes.

Finally, we highlighted the limitations of extending this method to 
concave polyominoes, where the row structure decomposes into sub-partitions 
with variable gaps, making the 4-connectivity constraint combinatorially 
intractable.

\end{document}